\documentclass[11pt,a4paper]{amsart}
\usepackage{hyperref}
\usepackage{amsfonts,amsmath,graphicx,float,graphicx,calc,enumerate}
\usepackage[numbers,sort&compress]{natbib}
\usepackage[lmargin=28mm,rmargin=28mm,tmargin=27mm,bmargin=27mm]{geometry}

\theoremstyle{plain}
\newtheorem{theorem}{Theorem}
\newtheorem{lemma}[theorem]{Lemma}
\newtheorem{corollary}[theorem]{Corollary}
\newtheorem{proposition}[theorem]{Proposition}
\newtheorem{conjecture}[theorem]{Conjecture}

\newcommand{\thmlabel}[1]{\label{thm:#1}}
\newcommand{\thmref}[1]{Theorem~\ref{thm:#1}}
\newcommand{\lemlabel}[1]{\label{lem:#1}}
\newcommand{\lemref}[1]{Lemma~\ref{lem:#1}}
\newcommand{\conjlabel}[1]{\label{con:#1}}
\newcommand{\conjref}[1]{Conjecture~\ref{con:#1}}

\newcommand{\figlabel}[1]{\label{fig:#1}}
\newcommand{\figref}[1]{Figure~\ref{fig:#1}}

\newcommand{\twofigref}[2]{Figures~\ref{fig:#1} and \ref{fig:#2}}
\newcommand{\seclabel}[1]{\label{sec:#1}}
\newcommand{\secref}[1]{Section~\ref{sec:#1}}
\newcommand{\corlabel}[1]{\label{cor:#1}}
\newcommand{\corref}[1]{Corollary~\ref{cor:#1}}
\newcommand{\proplabel}[1]{\label{prop:#1}}
\newcommand{\propref}[1]{Proposition~\ref{prop:#1}}

\newcommand{\0}[1]{}

\newcommand{\midx}{\delta}

\newcommand{\spacing}[1]{\renewcommand{\baselinestretch}{#1}\setlength{\footnotesep}{\baselinestretch\footnotesep}}
\spacing{1.2}

\newcommand{\Z}{\ensuremath{\mathbb{Z}}}
\newcommand{\R}{\ensuremath{\mathbb{R}}}

\newcommand{\Figure}[4][htb]{
\begin{figure}[#1]
	\vspace*{1ex}
	\begin{center}#3\end{center}
	\vspace*{-1ex}
	\caption{\figlabel{#2}#4}
\end{figure}
}

\newcommand{\vvv}[1]{(#1)}
\newcommand{\half}{\ensuremath{\protect\tfrac{1}{2}}}
\newcommand{\ceil}[1]{\ensuremath{\protect\lceil#1\rceil}}

\DeclareMathOperator{\ES}{ES}
\DeclareMathOperator{\col}{col}
\DeclareMathOperator{\conv}{conv}

\begin{document}

\title{Blocking Coloured Point Sets}

\thanks{This is the full version of an extended abstract to
    appear in the 26th European Workshop on Computational Geometry
    (EuroCG '10).}

\author[]{Greg Aloupis}
\address{Institute of Information Science, Academia     Sinica\newline
  Taipei, Taiwan}
\email{aloupis.greg@gmail.com}

\author[]{Brad Ballinger}
\address{Department of Mathematics, 
Humboldt State University\newline 
Arcata, California, U.S.A}
\email{bjb86@humboldt.edu}

\author[]{S\'ebastien Collette}
\address{Charg\'e de Recherches du F.R.S.-FNRS\newline D\'epartement d'Informatique,  Universit\'e Libre de Bruxelles\newline Brussels, Belgium}
\email{sebastien.collette@ulb.ac.be}

\author[]{Stefan Langerman}
\address{Ma\^itre de Recherches du F.R.S.-FNRS\newline D\'epartement d'Informatique,  Universit\'e Libre de Bruxelles\newline Brussels, Belgium}
\email{stefan.langerman@ulb.ac.be}

\author[]{Attila P\'or}
\address{Department of Mathematics, 
  Western Kentucky University
\newline Bowling Green,  Kentucky, U.S.A.}
\email{attila.por@wku.edu}

\author[]{David~R.~Wood}
\address{QEII Research Fellow\newline 
Department of Mathematics and Statistics, 
 The University of Melbourne
\newline Melbourne, Australia}
\email{woodd@unimelb.edu.au}


\begin{abstract}
This paper studies problems related to visibility among points in the plane.  
A point $x$ \emph{blocks} two points $v$ and $w$ if $x$ is in the
interior of the line segment $\overline{vw}$. 
A set of points $P$ is \emph{$k$-blocked} if each point in $P$ is
assigned one of $k$ colours, such that distinct points $v,w\in P$ are
assigned the same colour if and only if some other point in $P$ blocks
$v$ and $w$. The focus of this paper is the conjecture that each
$k$-blocked set has bounded size (as a function of $k$). Results in
the literature imply that every 2-blocked set has at most 3 points,
and every $3$-blocked set has at most 6 points. We prove that  every $4$-blocked set
has at most 12 points, and that this bound is tight.  In fact, we characterise all sets
$\{n_1,n_2,n_3,n_4\}$ such that some 4-blocked set has exactly $n_i$
points in the $i$-th colour class. Amongst other results, for
infinitely many values of $k$, we construct $k$-blocked sets with
$k^{1.79\ldots}$ points.
\end{abstract}

\maketitle


\newpage
\section{Introduction}
\seclabel{Intro}

This paper studies problems related to visibility and blocking in sets
of coloured points in the plane.  A point $x$ \emph{blocks} two points
$v$ and $w$ if $x$ is in the interior of the line segment
$\overline{vw}$. Let $P$ be a finite set of points in the plane. Two
points $v$ and $w$ are \emph{visible} with respect to $P$ if no point
in $P$ blocks $v$ and $w$.  The \emph{visibility graph} of $P$
has vertex set $P$, where two distinct points $v,w\in P$ are adjacent
if and only if they are visible with respect to $P$. A point set $B$ \emph{blocks}
$P$ if $P\cap B=\emptyset$ and for all distinct $v,w\in P$ there is a
point in $B$ that blocks $v$ and $w$. That is, no two points in $P$
are visible with respect to $P\cup B$, or alternatively, $P$ is an
independent set in the visibility graph of $P\cup B$.

A set of points $P$ is \emph{$k$-blocked} if each point in $P$ is
assigned one of $k$ colours, such that each pair of points $v,w\in P$
are visible with respect to $P$ if and only if $v$ and $w$ are
coloured differently. Thus $v$ and $w$ are assigned the same colour if
and only if some other point in $P$ blocks $v$ and $w$. A
\emph{$k$-set} is a multiset of $k$ positive
integers. For a $k$-set $\{n_1,\dots,n_k\}$, we say $P$ is
$\{n_1,\dots,n_k\}$-blocked if it is $k$-blocked and for some
labelling of the colours by the integers $[k]:=\{1,2,\dots,k\}$, the
$i$-th colour class has exactly $n_i$ points, for each
$i\in[k]$. Equivalently, $P$ is $\{n_1,\dots,n_k\}$-blocked if the
visibility graph of $P$ is the complete $k$-partite graph
$K(n_1,\dots,n_k)$.  A $k$-set $\{n_1,\dots,n_k\}$ is
\emph{representable} if there is an $\{n_1,\dots,n_k\}$-blocked point
set. See \figref{K3333} for an example.


\Figure[!ht]{K3333}{\includegraphics{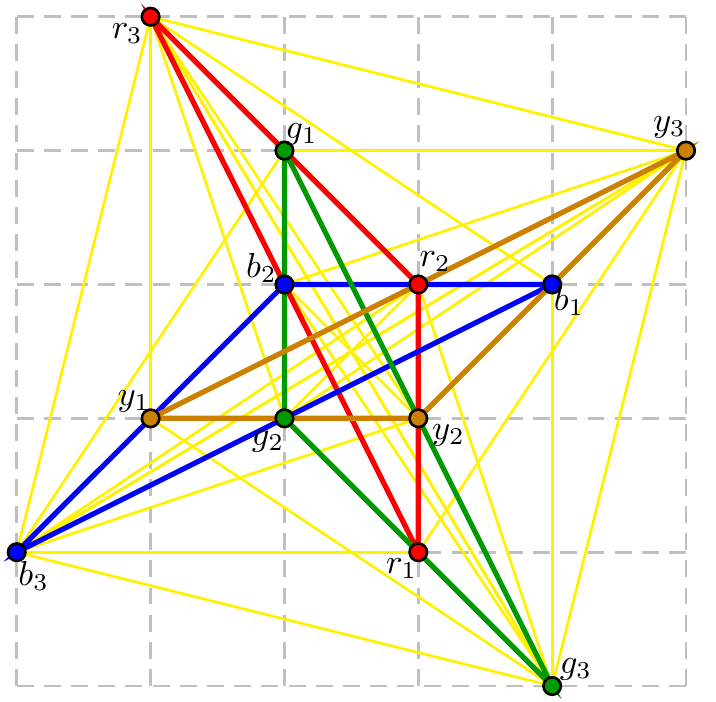}}{A $\{3,3,3,3\}$-blocked point set.}

The following fundamental conjecture regarding $k$-blocked point sets
is the focus of this paper. 

\begin{conjecture}
  \conjlabel{kBlocked} For each integer $k$ there  is an
  integer $n$ such that every $k$-blocked set has at most $n$ points.
\end{conjecture}

As illustrated in
\figref{23Blocked}, the following theorem is a direct consequence of the characterisation
of 2- and 3-colourable visibility graphs by \citet{KPW-DCG05}.

\begin{theorem}
\thmlabel{23} 
$\{1,1\}$ and $\{1,2\}$ are the only representable 2-sets, and \\
$\{1,1,1\}$, $\{1,1,2\}$, $\{1,2,2\}$ and $\{2,2,2\}$ are the only representable 3-sets. 
\end{theorem}

\Figure{23Blocked}{\includegraphics{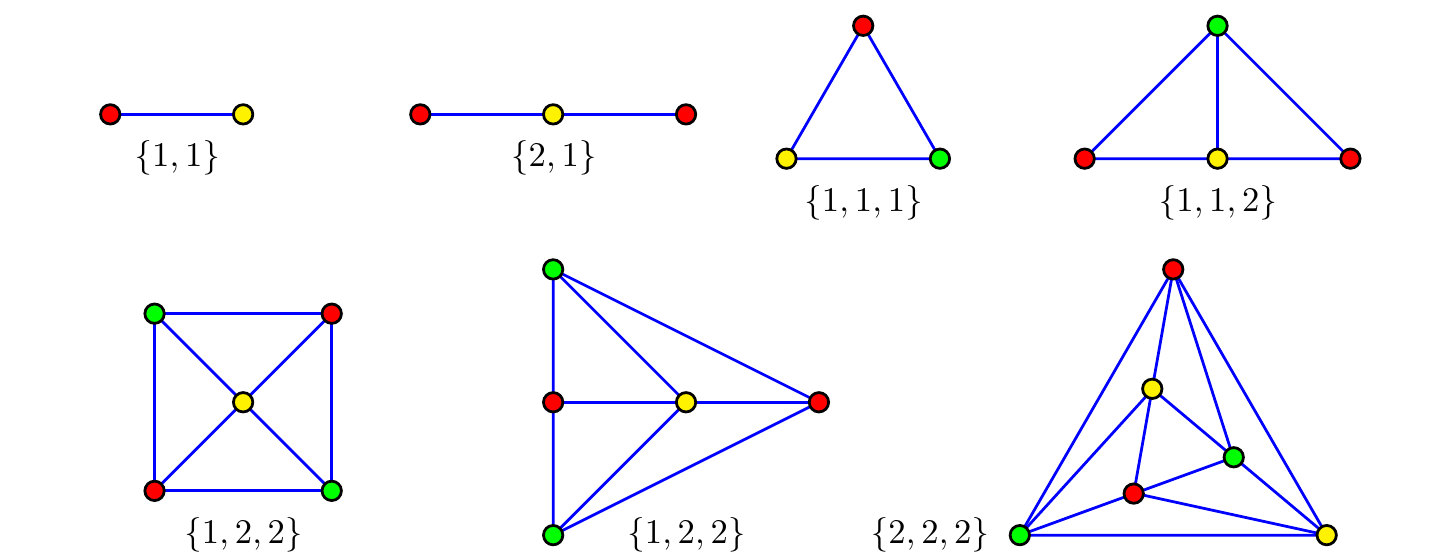}}{The 2-blocked and
  3-blocked point sets.}

In particular, every $2$-blocked point set has at most 3 points, and
every $3$-blocked point set has at most 6 points. This proves
\conjref{kBlocked} for $k\leq 3$.

This paper makes the following contributions. \secref{Background}
introduces some background motivation for the study of $k$-blocked
point sets, and observes that results in the literature prove
\conjref{kBlocked} for $k=4$. \secref{SmallColourClasses} describes
methods for constructing $k$-blocked sets from a given $(k-1)$-blocked
set. These methods lead to a characterisation of representable
$k$-sets when each colour class has at most three
points. \secref{Four} studies the $k=4$ case in more detail. In
particular, we characterise the representable 4-sets, and conclude
that the example in \figref{K3333} is in fact the largest
4-blocked point set. \secref{Midpoint} introduces a special class of
$k$-blocked sets (so-called midpoint-blocked sets) that lead to a
construction of the largest known $k$-blocked sets for infinitely many
values of $k$.

\section{Basic Properties}

\begin{lemma}
  \lemlabel{SelfBlockingThree} At most three points are collinear in
  every $k$-blocked point set.
\end{lemma}

\begin{proof}
  Suppose that four points  $p,q,r,s$ are collinear in this order. Thus $(p,q,r,s)$ is an induced path  in the visibility graph.  Thus $p$ is not adjacent to $r$ and not
  adjacent to $s$.  Thus $p$, $r$, and $s$ have the same colour.  This
  is a contradiction since $r$ and $s$ are adjacent.  Thus no four
  points are collinear.
\end{proof}

A set of points $P$ is in \emph{general position} if no three
points in $P$ are collinear.

\begin{lemma}
  \lemlabel{GenPos} Each colour class in a $k$-blocked
  point set is in general position.
\end{lemma}

\begin{proof}
  Suppose on the contrary that three points from a single colour
  class are collinear.  Then no other points are in the same line by
  \lemref{SelfBlockingThree}.  Thus two of the three points are
  adjacent, which is a contradiction.
\end{proof}

\section{Some Background Motivation}
\seclabel{Background}

Much recent research on blockers began with the following conjecture by
\citet{KPW-DCG05}.

\begin{conjecture}[Big-Line-Big-Clique Conjecture \citep{KPW-DCG05}]
  \conjlabel{BigClique} For all integers $t$ and $\ell$ there is an
  integer $n$ such that for every finite set $P$ of at least $n$
  points in the plane:
  \begin{itemize}
  \item $P$ contains $\ell$ collinear points, or
  \item $P$ contains $t$ pairwise visible points\\ 
(that is, the visibility graph of $P$ contains a $t$-clique).
  \end{itemize}
\end{conjecture}

\conjref{BigClique} is true for $t\leq 5$, but is open for $t\geq6$ or
$\ell\geq4$; see \citep{PorWood-Blockers,EmptyPentagon}. Given that, in general,  \conjref{BigClique} is challenging, Jan K{\'a}ra suggested the following weakening. 

\begin{conjecture}[\citep{PorWood-Blockers}]
\conjlabel{BigChromaticNumber} 
For all integers $t$ and $\ell$ there is an integer $n$ such that for every
finite set $P$ of at least $n$ points in the plane: 
  \begin{itemize}
  \item $P$ contains $\ell$ collinear points, or
  \item the chromatic number of the visibility graph of $P$ is at least $t$.
 \end{itemize}
\end{conjecture}

\conjref{BigClique} implies \conjref{BigChromaticNumber} since every
graph that contains a $t$-clique has chromatic number at least $t$. 

\begin{proposition}
\conjref{BigChromaticNumber} with $\ell=4$ and $t=k+1$ implies \conjref{kBlocked}.
\end{proposition}

\begin{proof}
  Assume \conjref{BigChromaticNumber} holds for $\ell=4$ and
  $t=k+1$. Suppose there is a $k$-blocked set $P$ of at least $n$
  points. By \lemref{SelfBlockingThree}, at most three points are
  collinear in $P$. Thus the first conclusion of
  \conjref{BigChromaticNumber} does not hold. By definition, the
  visibility graph of $P$ is $k$-colourable.  Thus the second conclusion of
  \conjref{BigChromaticNumber} does not hold. This contradiction
  proves that every $k$-blocked set has less than $n$ points, and
  \conjref{kBlocked} holds.
\end{proof}

Thus, since \conjref{BigClique} holds for $t\leq5$, 
\conjref{kBlocked} holds for $k\leq4$. In \secref{Four} we take this
result much further, by characterising all representable 4-sets,
this concluding a tight bound on the size of a 4-blocked set.

\medskip Part of our interest in blocked point sets comes from the
following.

\begin{proposition}
  \proplabel{TuranGraph} For all $k\geq3$ and $n\geq2$, the edge set
  of the $k$-partite Tur\'an graph $K(n,n,\dots,n)$ can be partitioned
  into a set of `lines', where:
  \begin{itemize}
  \item each line is either an edge or an induced path on three
    vertices,
  \item every pair of vertices is in exactly one line, and
  \item for every line $L$ there is a vertex adjacent
    to each vertex in $L$.
  \end{itemize}
\end{proposition}

\begin{proof} 
Let $\vvv{i,p}$ be the $p$-th vertex in the
  $i$-th colour class for $i\in\Z_k$ and $p\in\Z_n$ (taken as additive
  cyclic groups).  We introduce three types of lines.  First, for
  $i\in\Z_k$ and distinct $p,q\in\Z_n$, let the triple
  $\{\vvv{i,p},\vvv{i+1,p+q},\vvv{i,q}\}$ be a line. Second, for
  $i\in\Z_k$ and $p\in\Z_n$, let the pair
  $\{\vvv{i,p},\vvv{i+1,p+p}\}$ be a line.  Third, for $p,q\in\Z_n$
  and distinct non-consecutive $i,j\in\Z_k$, let
  the pair $\{\vvv{i,p},\vvv{j,q}\}$ be a line.  By construction each
  line is either an edge or an induced path on three vertices.

  Every pair of vertices in the same colour class are in exactly one
  line (of the first type). Consider vertices $\vvv{i,p}$ and
  $\vvv{j,q}$ in distinct colour classes. First suppose that $i$ and
  $j$ are consecutive.  Without loss of generality, $j=i+1$.  If
  $q\neq p+p$ then $\vvv{i,p}$ and $\vvv{j,q}$ are in exactly one line
  (of the first type).  If $q=p+p$ then $\vvv{i,p}$ and $\vvv{j,q}$
  are in exactly one line (of the second type). If $i$ and $j$ are not
  consecutive, then $\vvv{i,p}$ and $\vvv{j,q}$ are in exactly one
  line (of the third type).  This proves that every pair of vertices
  is in exactly one line.  Moreover, every edge of $K(n,n,\dots,n)$ is
  in exactly one line, and the lines partition the edges set.

  Since every line $L$ is contained in the union of two colour
  classes, each vertex in neither colour class
  intersecting $L$ is adjacent to each vertex in $L$.
\end{proof}

\propref{TuranGraph} is significant for \conjref{BigClique} because it
says that $K(n,\dots,n)$ behaves like a `visibility space' with no
four collinear points, but it has no large clique (for fixed
$k$). \conjref{kBlocked}, if true, implies that such a visibility
space is not `geometrically representable' for large $n$.

\medskip
Let $b(n)$ be the minimum integer such that some set of $n$ points in
the plane in general position is blocked by some set of $b(n)$
points. \citet{Matousek09} proved that $b(n)\geq 2n-3$. 
\citet{DPT09} improved this bound to $b(n)\geq(\frac{25}{8}-o(1))n$.
Many authors have conjectured or stated as an open problem that
$b(n)$ is super-linear.

\begin{conjecture}[\citep{Matousek09,Pinchasi,DPT09,PorWood-Blockers}]
\conjlabel{Blockers}
$\frac{b(n)}{n}\rightarrow\infty$ as $n\rightarrow\infty$.
\end{conjecture}

\citet{PorWood-Blockers} proved that  \conjref{Blockers}
implies \conjref{BigChromaticNumber}, and thus implies
\conjref{kBlocked}. That  \conjref{kBlocked} is implied by a number of other
well-known conjectures, yet remains challenging, adds to its interest.

\section{$k$-Blocked Sets with Small Colour Classes}
\seclabel{SmallColourClasses}

We now describe some methods for building blocked point sets
from smaller blocked point sets.

\begin{lemma}
  \lemlabel{BuildSelfBlocking} Let $G$ be a visibility graph. Let
  $i\in\{1,2,3\}$.  Furthermore suppose that if $i\geq2$ then
  $V(G)\neq\emptyset$, and if $i=3$ then not all the vertices of $G$
  are collinear.  Let $G_i$ be the graph obtained from $G$ by adding
  an independent set of $i$ new vertices, each adjacent to every
  vertex in $G$.  Then $G_1$, $G_2$, and $G_3$ are visibility graphs.
\end{lemma}

\begin{proof}
  For distinct points $p$ and $q$, let $\overleftarrow{pq}$ denote the ray
  that is (1) contained in the line through $p$ and $q$, (2) starting
  at $p$, and (3) not containing $q$. Let $\mathcal{L}$ be the union
  of the set of lines containing at least two vertices in $G$.

  $i=1$: Since $\mathcal{L}$ is the union of finitely many lines,
  there is a point $p\not\in \mathcal{L}$. Thus $p$ is visible from
  every vertex of $G$. By adding a new vertex at $p$, we obtain a
  representation of $G_1$ as a visibility graph.

  $i=2$: Let $p$ be a point not in $\mathcal{L}$. Let $v$ be a vertex
  of $G$.  Each line in $\mathcal{L}$ intersects $\overleftarrow{vp}$
  in at most one point.  Thus
  $\overleftarrow{vp}\setminus\mathcal{L}\neq\emptyset$.  Let $q$ be a
  point in $\overleftarrow{vp}\setminus\mathcal{L}$.  Thus $p$ and $q$
  are visible from every vertex of $G$, but $p$ and $q$ are blocked by
  $v$. By adding new vertices at $p$ and $q$, we obtain a
  representation of $G_2$ as a visibility graph.


  $i=3$: Let $u,v,w$ be non-collinear vertices in $G$.  Let $p$ be a
  point not in $\mathcal{L}$ and not in the convex hull of $\{u,v,w\}$. Without loss
  of generality, $\overline{uv}\cap\overline{pw}\neq\emptyset$.  There
  are infinitely many pairs of points $q\in\overleftarrow{up}$ and
  $r\in\overleftarrow{vp}$ such that $w$ blocks $q$ and $r$. Thus
  there are such $q$ and $r$ both not in $\mathcal{L}$.  By
  construction, $u$ blocks $p$ and $q$, and $v$ blocks $p$ and $r$. By
  adding new vertices at $p$, $q$ and $r$, we obtain a representation
  of $G_3$ as a visibility graph.
\end{proof}

Since no $(\geq3)$-blocked set is collinear,   \lemref{BuildSelfBlocking}  implies:

\begin{corollary} 
  \corlabel{Build} If $k\geq4$ and   $\{n_1,\dots,n_{k-1}\}$ is
  representable and $n_k\in\{1,2,3\}$, then
  $\{n_1,\dots,n_{k-1},n_k\}$ is representable.
 \end{corollary}


The representable $(\leq 3)$-sets were characterised in \thmref{23}. In
each case, each colour class has at most three vertices. Now we
characterise the representable $(\geq 4)$-sets, assuming that
each colour class has at most three vertices.

\begin{proposition}
\proplabel{SmallColourClasses}
$\{n_1,\dots,n_k\}$ is representable whenever $k\geq4$ and each $n_i\leq 3$, except for $\{1,3,3,3\}$. 
\end{proposition}

\begin{proof}
  We say the $k$-set $\{n_1,\dots,n_k\}$ \emph{contains}
the $(k-1)$-set  $\{n_1,\dots,n_{i-1},n_{i+1},\dots,n_k\}$ for each $i\in[k]$.  
\lemref{1333} below proves that $\{1,3,3,3\}$ is not representable. We
  proceed by induction on $k$. If $\{n_1,\dots,n_k\}$ contains a
  representable $(k-1)$-set, then \corref{Build} implies
  that $\{n_1,\dots,n_k\}$ is also representable. Now assume that
  every $(k-1)$-set contained in $\{n_1,\dots,n_k\}$ is not
  representable.  By induction, we may assume that $k\leq 5$.
  Moreover, if $k=5$ then $\{n_1,\dots,n_5\}$ must contain $\{1,3,3,3\}$
  (since by induction all other 4-sets are representable).
  Similarly, if $k=4$ then $\{n_1,\dots,n_4\}$ must contain $\{1,1,3\}$,
  $\{1,2,3\}$, $\{1,3,3\}$, $\{2,2,3\}$, $\{2,3,3\}$ or $\{3,3,3\}$ (since
  $\{1,1,1\}$, $\{1,1,2\}$, $\{1,2,2\}$ and $\{2,2,2\}$ are
  representable by \thmref{23}). The following table describes the required
  construction in each remaining case.

  \bigskip
  \begin{tabular}{ll|ll}
    \hline
    $\{1,1,1,x\}$ & contains $\{1,1,1\}$ & 
    $\{1,1,2,x\}$ & contains $\{1,1,2\}$\\
    $\{1,1,3,3\}$ & \figref{K3333} minus $\{r_1,g_3,r_3,g_1\}$& 
    $\{1,2,2,x\}$ & contains $\{1,2,2\}$\\
    $\{1,2,3,3\}$ & \figref{K3333} minus $\{g_1,g_3,r_3\}$& 
    $\{2,2,2,x\}$ & contains $\{2,2,2\}$\\
    $\{2,2,3,3\}$ & \figref{K3333} minus $\{g_3,r_3\}$ &
    $\{2,3,3,3\}$ & \figref{K3333} minus $g_3$\\
    $\{1,1,3,3,3\}$ & contains $\{1,1,3,3\}$ & 
    $\{1,2,3,3,3\}$ & contains $\{1,2,3,3\}$\\
    $\{1,3,3,3,3\}$ & contains $\{3,3,3,3\}$\\
    \hline
  \end{tabular}
  \newline
\end{proof}


\section{4-Blocked Point Sets}
\seclabel{Four}

As we saw in \secref{Background}, \conjref{kBlocked} holds for
$k\leq4$. In this section we study 4-blocked point sets in more
detail. First we derive explicit bounds on the size of 4-blocked sets
from other results in the literature. Then, following a more detailed
approach, we characterise all representable 4-sets, to conclude a
tight bound on the size of 4-blocked sets.

\begin{proposition}
  \proplabel{4BlockedA} Every $4$-blocked set has less than $2^{790}$
  points.
\end{proposition}

\begin{proof}
  \citet{EmptyPentagon} proved that every set of at least
  $\ES(\frac{(2\ell-1)^{\ell}-1}{2\ell-2} )$ points in the plane
  contains $\ell$ collinear points or an empty convex pentagon, where
  $\ES(k)$ is the minimum integer such that every set of at least
  $\ES(k)$ points in general position in the plane contains $k$ points
  in convex position.  Let $P$ be a $4$-blocked set.  The visibility
  graph of $P$ is $4$-colourable, and thus contains no empty convex
  pentagon.  By \lemref{SelfBlockingThree}, at most three points in
  $P$ are collinear. Thus $|P| \leq \ES(400)-1$ by the above result
  with $\ell=4$.  \citet{TothValtr05} proved that
  $\ES(k)\leq\binom{2k-5}{k-2}+1$.  Hence $|P|\leq
  \binom{795}{398}<2^{790}$.
\end{proof}


\begin{lemma}
  \lemlabel{SelfBlockingLargeColourClass} If $P$ is a blocked set of
  $n$ points with $m$ points in the largest colour class, then
  $n\geq 3m-3$ and $n\geq(\frac{33}{8}-o(1))m$.
\end{lemma}

\begin{proof}
  If $S$ is the largest colour class then $P-S$ blocks $S$.  By
  \lemref{GenPos}, $S$ is in general position.  By the results of
  \citet{Matousek09} and \citet{DPT09} mentioned in
  \secref{Background}, $n-m\geq 2m-3$ and
  $n-m\geq(\frac{25}{8}-o(1))m$.
\end{proof}

\begin{proposition}
  \proplabel{4BlockedB} Every $4$-blocked set has less than $2^{578}$
  points.
\end{proposition}

\begin{proof}
  Let $P$ be a $4$-blocked set of $n$ points.  Let $S$ be the largest
  colour class in $P$.  Let $m: = |S| \geq\frac{n}{4}$.  By
  \lemref{SelfBlockingLargeColourClass},
  $n\geq(\frac{33}{8}-o(1))m\geq(\frac{33}{32}-o(1))n$.  Thus
  $o(n)\geq\frac{n}{32}$. Hence $n$ is bounded.  A precise bound is
  obtained as follows.  \citet{DPT09} proved that $S$ needs at least
  $\frac{25m}{8}-\frac{25m}{2\ln m}-\frac{25}{8}$ blockers, which come
  from the other colour classes.  Thus $n-m \geq
  \frac{25m}{8}-\frac{25m}{2\ln m}-\frac{25}{8}$, implying $n \geq
  \frac{33m}{8}-\frac{25m}{2\ln m}-\frac{25}{8}$.  Since
  $\frac{n}{4}\leq m\leq n$, we have $\frac{25n}{2\ln n}+\frac{25}{8}
  \geq \frac{n}{32}$.  It follows that $n\leq 2^{578}$.
\end{proof}

The next result is the simplest known proof that every 4-blocked point
set has bounded size.

\begin{proposition}
  \proplabel{4BlockedC} Every $4$-blocked set has at most $36$ points.
\end{proposition}

\begin{proof}
  Let $P$ be a 4-blocked set. Suppose that $|P| \geq 37$.  Let $S$ be
  the largest colour class. Thus $|S|\geq10$.  By \lemref{GenPos}, $S$
  is in general position.  By a theorem of \citet{Harborth78}, some
  $5$-point subset $K\subseteq S$ is the vertex-set of an empty convex
  pentagon $\conv(K)$.  Let $T:= P \cap ( \conv(K) - K )$.  Since
  $\conv(K)$ is empty with respect to $S$, each point in $T$ is not in
  $S$.  Thus $T$ is 3-blocked.  $K$ needs at least 8 blockers (5
  blockers for the edges on the boundary of $\conv(K)$, and 3 blockers
  for the chords of $\conv(K)$).  Thus $|T| \geq 8$. But every
  3-blocked set has at most 6 points, which is a contradiction.  Hence
  $|P| \leq 36$.
\end{proof}

We now set out to characterise all representable 4-sets. We need a
few technical lemmas.

\begin{lemma}
  \lemlabel{New} 
Let $A$ be a set of three  monochromatic points in a
  4-blocked set $P$. Then $P\cap\conv(A)$ contains a point from  each colour class. 
\end{lemma}

\begin{proof}
  $P\cap\conv(A)$ contains at least three points in $A$.  If
  $P\cap\conv(A)$ contains no point from one of the three other colour
  classes, then $P\cap\conv(A)$ is a $(\leq 3)$-blocked set with three
  points in one colour class ($A$), contradicting \thmref{23}.
\end{proof}

\begin{lemma}
\lemlabel{1333}
$\{3,3,3,1\}$ is not representable.
\end{lemma}

\begin{proof}
Let $a*b*c$ mean that $b$ blocks $a$ and $c$, and let $a*b*c*d$ mean
that $a*b*c$ and $b*c*d$.  This allows us to record the order in which points occur on a line.

Suppose that $P$ is a $\{3,3,3,1\}$-blocked set of points, with colour classes $A=\{a_1, a_2, a_3\}$, $B=\{b_1, b_2, b_3\}$, $C=\{c_1, c_2, c_3\}$, and $D=\{d\}$.

If $d$ does not block some monochromatic pair, then $P\setminus D$ is
a 3-blocked set of nine points, contradicting \thmref{23}. Therefore $d$
blocks some monochromatic pair, which we may call $a_1, a_2$.  Now
$\overleftrightarrow{a_1a_2}$ divides the remaining seven points of
$P$ into two sets, $P_1$ and $P_2$, where WLOG $4\leq|P_1|\leq 7$ and $0\leq|P_2|\leq3$.  If $|P_1|\ge6$, then $P_1\cup\{a_1\}$ is a 3-blocked set of more
than six points, contradicting \thmref{23}.  Thus $|P_1|=4$ and $|P_2|=3$, or 
$|P_1|=5$ and $|P_2|=2$. Consider the following cases, as illustrated in \figref{1333}.

\textbf{Case 1.} $|P_1|=4$ and $|P_2|=3$:

We have $a_3 \in P_1$ and $P_1\neq\{a_3,b_1,b_2,b_3\}$ and
$P_1\neq\{a_3,c_1,c_2,c_3\}$, as otherwise $P_1$ is a 2-blocked set of
four points, contradicting \thmref{23}. WLOG, $P_1=\{a_3, b_1,
b_2,c_1\}$ and $P_2=\{b_3, c_2, c_3\}$, where $c_2*b_3*c_3$.  Some
point in $P_1$ blocks $b_1$ and $b_2$.  If $b_1*a_3*b_2$, then neither
$b_1$ nor $b_2$ can block $\overline{a_3a_1}$ or $\overline{a_3a_2}$; thus $c_1$ blocks
$a_3$ from both $a_1$ and $a_2$, a contradiction.  Therefore
$b_1*c_1*b_2$.  Since two of these three points must block $\overline{a_3a_1}$
and $\overline{a_3a_2}$, we may assume that $a_1*b_1*a_3$, and either $a_2*b_2*a_3$ or
$a_2*c_1*a_3$.

\begin{figure}[ht]
\includegraphics[scale=0.61]{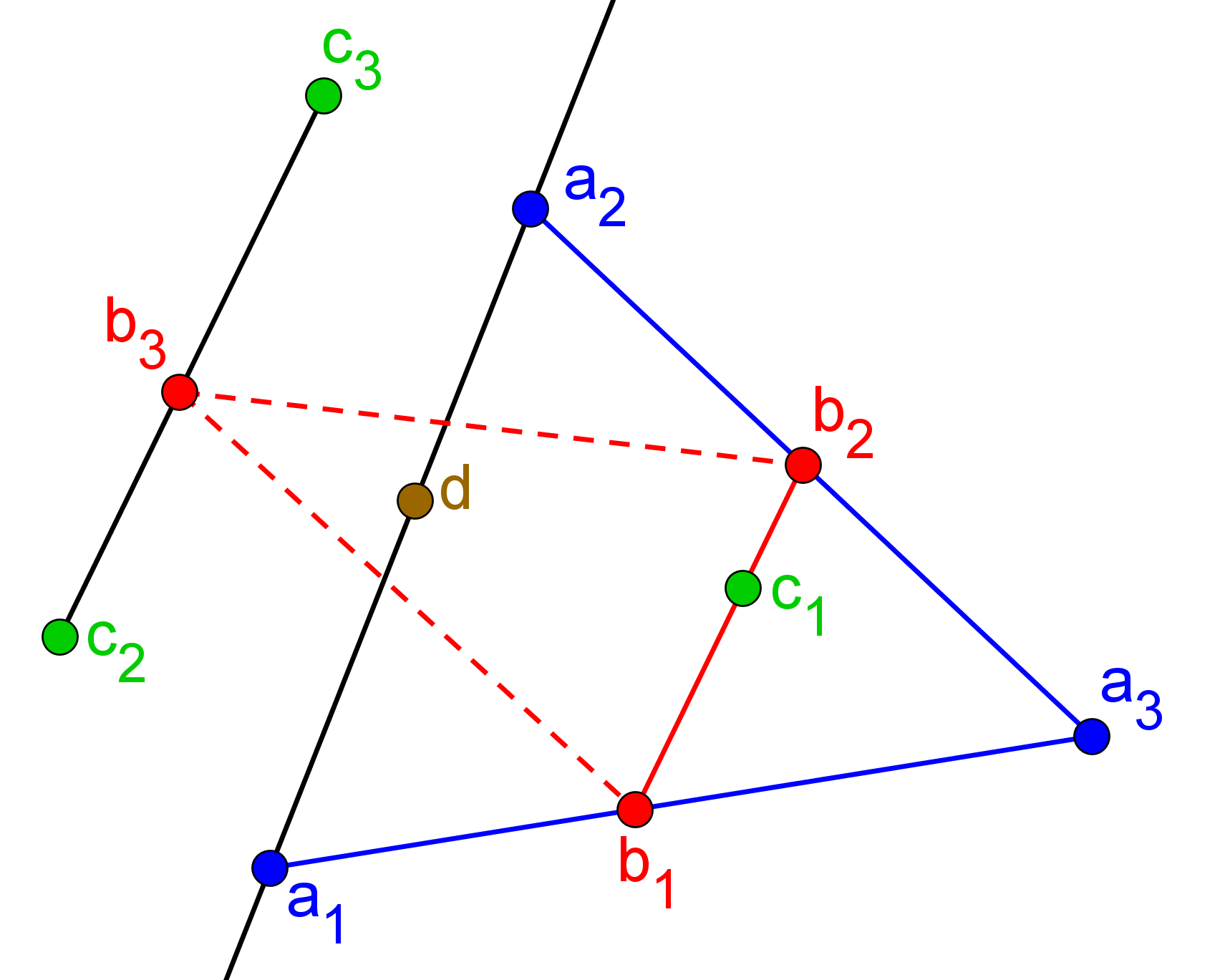}
\hfill
\includegraphics[scale=0.61]{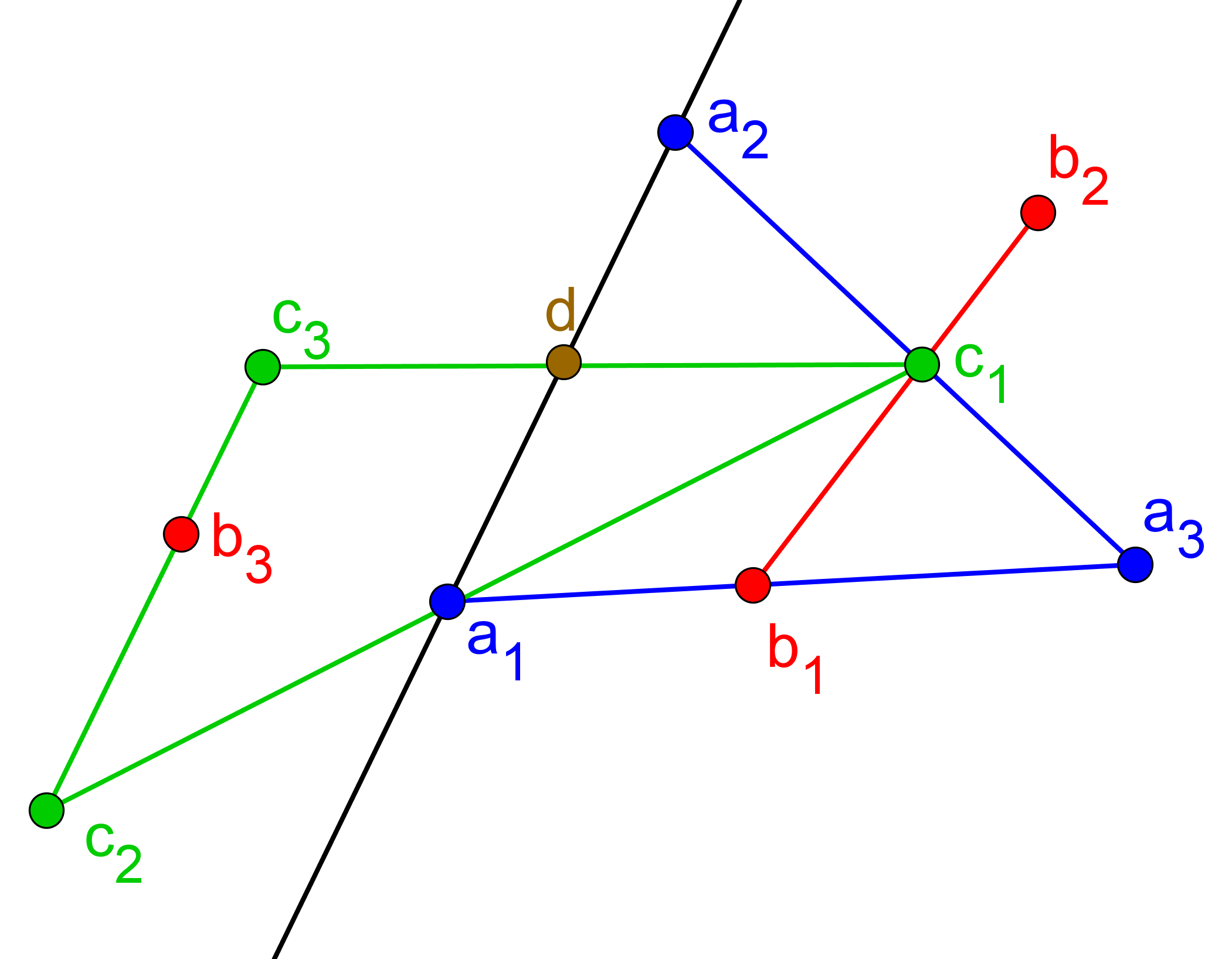}
\hfill
\includegraphics[scale=0.61]{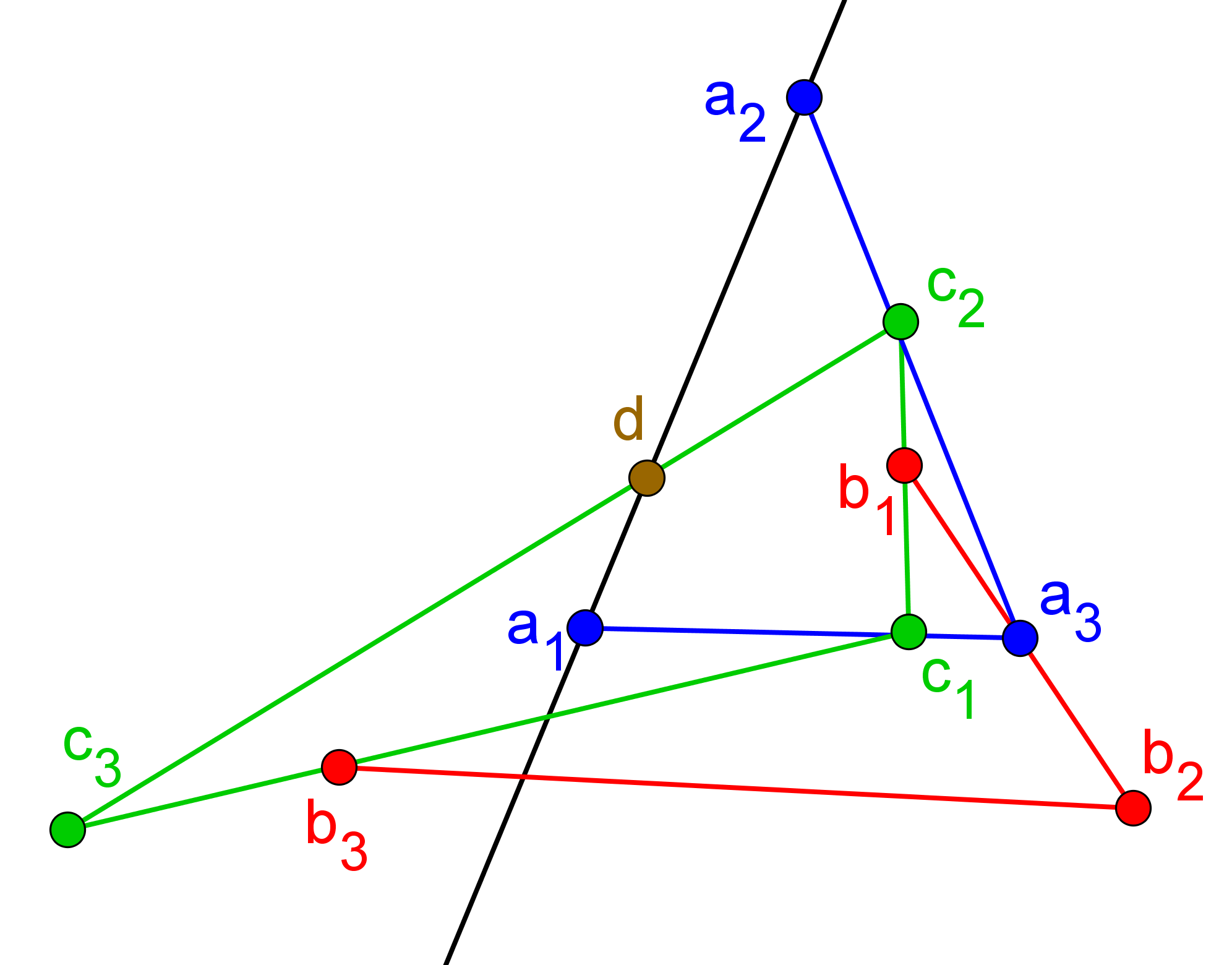}
\caption{\figlabel{1333} Cases 1.1, 1.2, and 2.}
\end{figure}

Case 1.1. $a_2*b_2*a_3$:  Since $c_2*b_3*c_3$, the only possible blockers for $\overline{b_1b_3}$ are on $\overleftrightarrow{a_1a_2}$.  We cannot have $b_3*a_1*b_1$, for then $b_3*a_1*b_1*a_3$.  We cannot have $b_3*a_2*b_1$, for then $\overleftrightarrow{b_1b_3}$ separates $b_2$ from $a_1$ and $d$, one of which needs to block $\overline{b_2b_3}$.  Therefore $b_1*d*b_3$.  Similarly, $b_2*d*b_3$, a contradiction.

Case 1.2. $a_2*c_1*a_3$:  Thus $a_2$ does not block
$\overline{c_1c_2}$ or $\overline{c_1c_3}$.  Therefore $a_1$ and $d$
block $\overline{c_1c_2}$ and $\overline{c_1c_3}$.  WLOG,
$c_2*a_1*c_1$ and $c_3*d*c_1$.  Since $c_2*b_3*c_3$, the blocker for
$\overline{b_1b_3}$ is on the same side of
$\overleftrightarrow{c_1c_3}$ as $a_1$, and on the same side of
$\overleftrightarrow{b_1b_2}$ as $a_1$.  The only such points are
$a_1$, $c_2$, and $b_3$.  Thus $a_1$ or $c_2$ block
$\overline{b_1b_3}$. Now, $c_2$ does not block $\overline{b_1b_3}$, as
otherwise $b_1*c_2*b_3*c_3$. Similarly, 
$a_1$ does not block $\overline{b_1b_3}$, as
otherwise $b_3*a_1*b_1*a_3$. This contradiction concludes this case.

\textbf{Case 2.} $|P_1|=5$ and $|P_2|=2$:

We have $a_3\in P_1$, as otherwise $P_1$ is 2-blocked, contradicting
\thmref{23}. WLOG,
$P_1=\{a_3, b_1, b_2, c_1, c_2\}$ and $P_2=\{b_3, c_3\}$.  Note that
$a_3$ must block at least one of $\overline{b_1b_2}$,
$\overline{c_1c_2}$, because $P_1\setminus \{a_3\}$ is too large to be
2-blocked; however, it cannot block both, for then there would be no
valid blockers left for $a_3$.  Thus WLOG $b_1*a_3*b_2$ and
$c_1*b_1*c_2$.  Since $a_3\in \overline{b_1b_2}$, neither $b_1$ nor
$b_2$ block $\overline{a_3a_1}$ or $\overline{a_3a_2}$.  Thus, WLOG, $a_1*c_1*a_3$ and $a_2*c_2*a_3$.  

By \lemref{New}, $P\cap\conv\{c_1,c_2,c_3\}$ contains some member of $A$, WLOG $a_1$.  We cannot have $a_1 \in \overline{c_1c_3}$, for then $c_3*a_1*c_1*a_3$.  Therefore $a_1$ is on the same side of $\overleftrightarrow{c_1c_3}$ as $c_2$; consequently, $d$ and $a_2$ are, as well.  It follows that $c_1*b_3*c_3$, and so $b_3$ sees $b_2$.
\end{proof}

\begin{lemma}
  \lemlabel{K4221} Let $P$ be a 4-blocked set. Suppose that some
  colour class $S$ of $P$ contains a subset $K$, such that $|K|=4$ and
  $K$ is the vertex-set of a convex quadrilateral $\conv(K)$ that is
  empty with respect to $S$. Then $P$ is $\{4,2,2,1)$-blocked.
\end{lemma}

\begin{proof}
  Let $T:= P \cap (\conv(K) - K )$.  Since $\conv(K)$ is empty with
  respect to $S$, each point in $T$ is not in $S$.  Thus $T$ is
  3-blocked.  $K$ needs at least 5 blockers (4 blockers for the edges
  on the boundary of $\conv(K)$, and at least 1 blocker for the chords
  of $\conv(K)$).  The only representable 3-sets with at least 5
  points are $\{2,2,1\}$ and $\{2,2,2\}$.  Exactly four points in $T$ are
  on the boundary of $\conv(T)$.  Every $\{2,2,2\}$-blocked set contains
  three points on the boundary of the convex hull.  Thus $T$ is
  $\{2,2,1\}$-blocked.  Hence, as illustrated in
  \figref{BradsTwistedGrid}, one point $c$ in $T$ is at the
  intersection of the two chords of $\conv(K)$, and exactly one point
  in $T$ is on each edge of the boundary of $\conv(K)$, such that the
  points on opposite edges of $\conv(K)$ are collinear with $c$.

  \begin{figure}[ht]
    \includegraphics[scale=0.8]{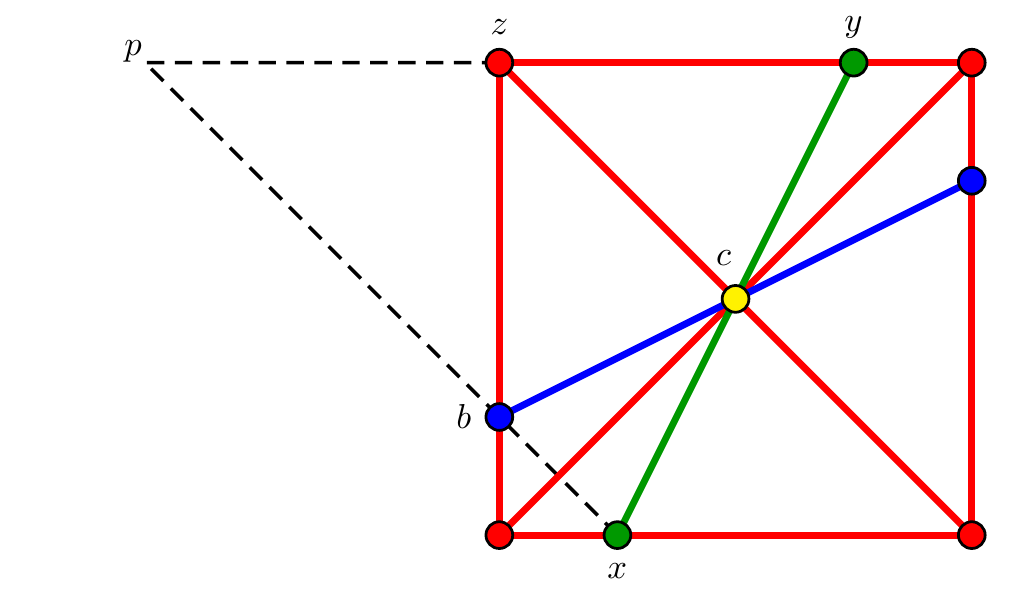}
    \caption{\figlabel{BradsTwistedGrid} A $\{4,2,2,1\}$-blocked point
      set.}
  \end{figure}

  We claim that no other point is in $P$.  Suppose otherwise, and let
  $p$ be a point in $P$ outside $\conv(K)$ at minimum distance from
  $\conv(K)$.  Let $x$ be a point in $\conv(K)$ receiving the same
  colour as $p$.  Thus $p$ and $x$ are blocked by some point $b$ in
  $\conv(K)$.  Thus $b$ and $x$ are collinear with no other point in
  $P\cap \conv(K)$.  Hence $x\neq c$ and $x\not\in K$.  Thus $x$ is in the
  interior of one of the edges of the boundary of $\conv(K)$.  Let $y$
  be the point in $\conv(K)$ receiving the same colour as $x$.  Thus
  $x$ and $y$ are on opposite edges of the boundary of $\conv(K)$.
  Hence $p$ and $y$ receive the same colour, implying $p$ and $y$ are
  blocked. Since $p$ is at minimum distance from $\conv(K)$, this
  blocker must a point $z$ of $\conv(K)$. This implies that $p,z,y$
  and one other point of $\conv(K)$ are four collinear points, which
contradicts \lemref{SelfBlockingThree}. Hence no other point is in $P$, and $P$ is
  $\{4,2,2,1\}$-blocked.
\end{proof}

\lemref{K4221} has the following corollary (let $K:=S$).

\begin{corollary}
  \corlabel{K4221} Let $P$ be a 4-blocked set. Suppose that some
  colour class $S$ consists of exactly four points in convex
  position. Then $P$ is $\{4,2,2,1\}$-blocked.
\end{corollary}

The next lemma is a key step in our characterisation of representable
4-sets.

\begin{lemma}
  \lemlabel{FourInColourClass} Each colour class in a 4-blocked point
  set has at most four points.
\end{lemma}

\begin{proof}
  Suppose that some $4$-blocked point set $P$ has a colour class $S$
  with at least five points. Esther Klein~\citep{ES35} proved that
  every set of least five points in general position in the plane
  contains an empty quadrilateral. By \lemref{GenPos}, $S$ is in
  general position. Thus $S$ contains a subset $K$, such that $|K|=4$
  and $K$ is the vertex-set of a convex quadrilateral $\conv(K)$ that
  is empty with respect to $S$. By \lemref{K4221}, $P$ is
  $\{4,2,2,1\}$-blocked, which is the desired contradiction.
\end{proof}


\begin{lemma}
  \lemlabel{K4222} Let $P$ be a 4-blocked point set with colour classes
  $A, B, C, D$. Suppose that no colour class consists of exactly
  four points in convex position (that is, \corref{K4221} is not
  applicable). Furthermore, suppose that some colour class $A$
  consists of exactly four points in nonconvex position. Then $P$ is
  $\{4,2,2,2\}$-blocked (as in \figref{K4222}).
\end{lemma}

\begin{figure}[ht]
  \includegraphics{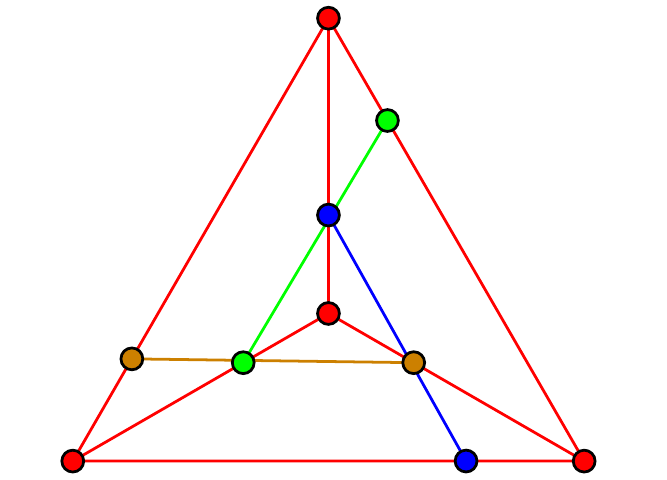}
  \caption{\figlabel{K4222} A $\{4,2,2,2\}$-blocked point set.}
\end{figure}

\begin{proof}
  By \lemref{FourInColourClass}, each colour class has at most four
  points.  By assumption, every 4-point colour class is in nonconvex
  position. We may assume that $A$ is a minimal in the sense that no
  other 4-point colour class is within $\conv(A)$.

  Let $Q:=P\cap \conv(A)$. Thus $Q$ is 4-blocked, and one colour
  class is $A$. By the minimality of $A$, each other colour class in
  $Q$ has at most three points. We first prove that $Q$ is
  $\{4,2,2,2\}$-blocked, and then show that this implies that $P=Q$.

  Let $A=\{a_1, a_2, a_3, a_4\}$, where $a_4$ is the  interior point of
  $\conv(A)$.  Note that the edges with points in $A$ divide
  $\conv (A)$ into three triangles with disjoint interiors. By
  \lemref{New}, each
  colour class of $Q$ is represented in each of these triangles;
  this requires at least two points of each colour (one of which
  could sit on the edge shared by two triangles).

  We name a point with reference to its colour class, such as $b_1\in
  B$; or, we name a point with reference to its position.  Let
  $x_{ij}$ be the unique member of $Q':=Q\setminus A$ that blocks
  $\overline{a_ia_j}$, for $1\le i<j \le 4$.  This accounts for
  exactly 6 points of $Q'$.

  $Q$ is not $\{4,3,3,3\}$-blocked, as otherwise we could delete the three
  outer members of $A$ to represent $\{3,3,3,1\}$, which contradicts
  \lemref{1333}. Thus $Q$ is $\{4,2,2,2\}$-blocked,
  $\{4,3,2,2\}$-blocked, or $\{4,3,3,2\}$-blocked.

  First suppose that $Q$ is $\{4,3,2,2\}$-blocked. Then $Q'$ consists of
  six points $x_{ij}$ and one additional point, $y$.  We have three
  cases:

  \begin{figure}[ht]
    \includegraphics[scale=0.69]{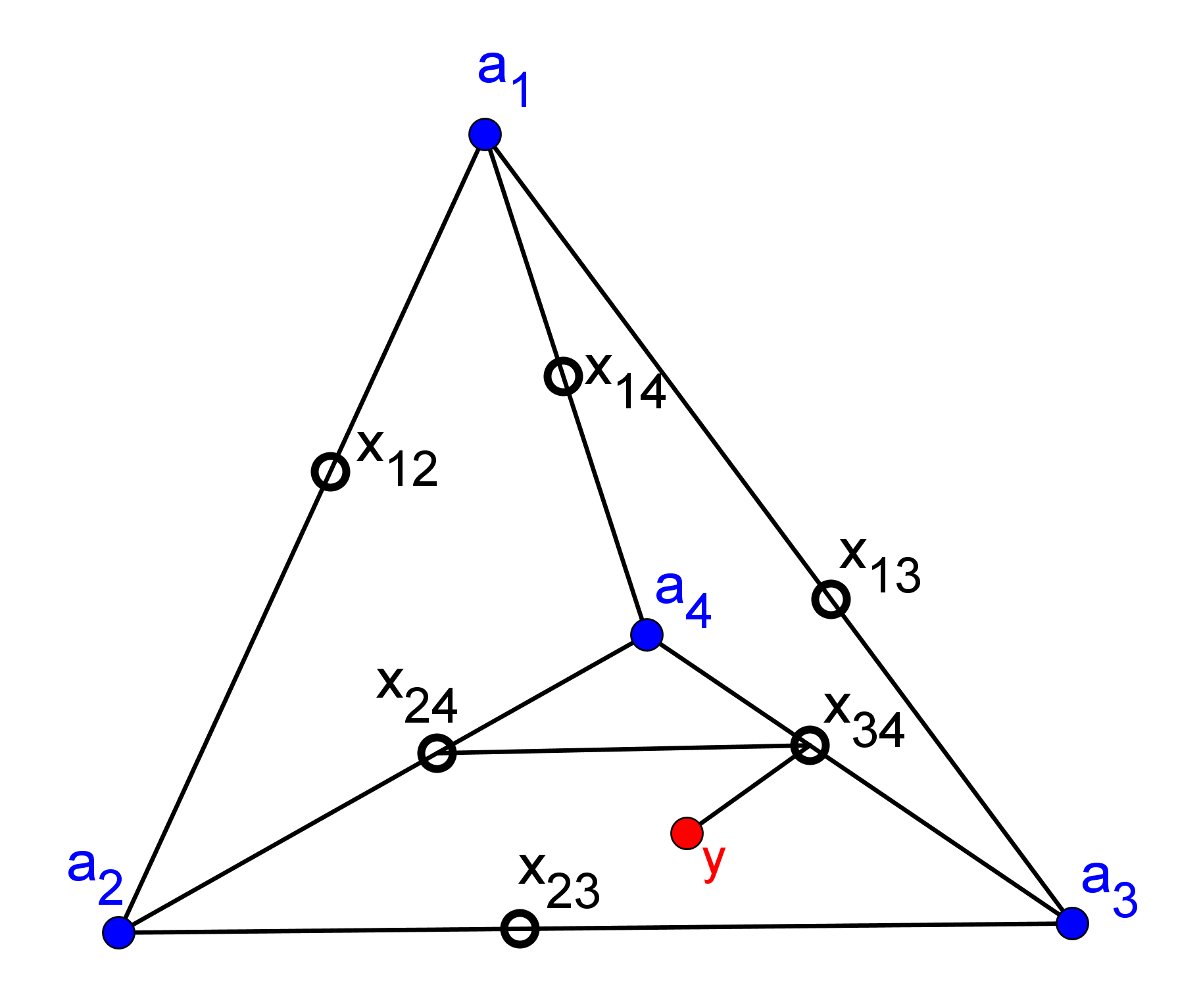}
\hspace*{-4mm}
    \includegraphics[scale=0.69]{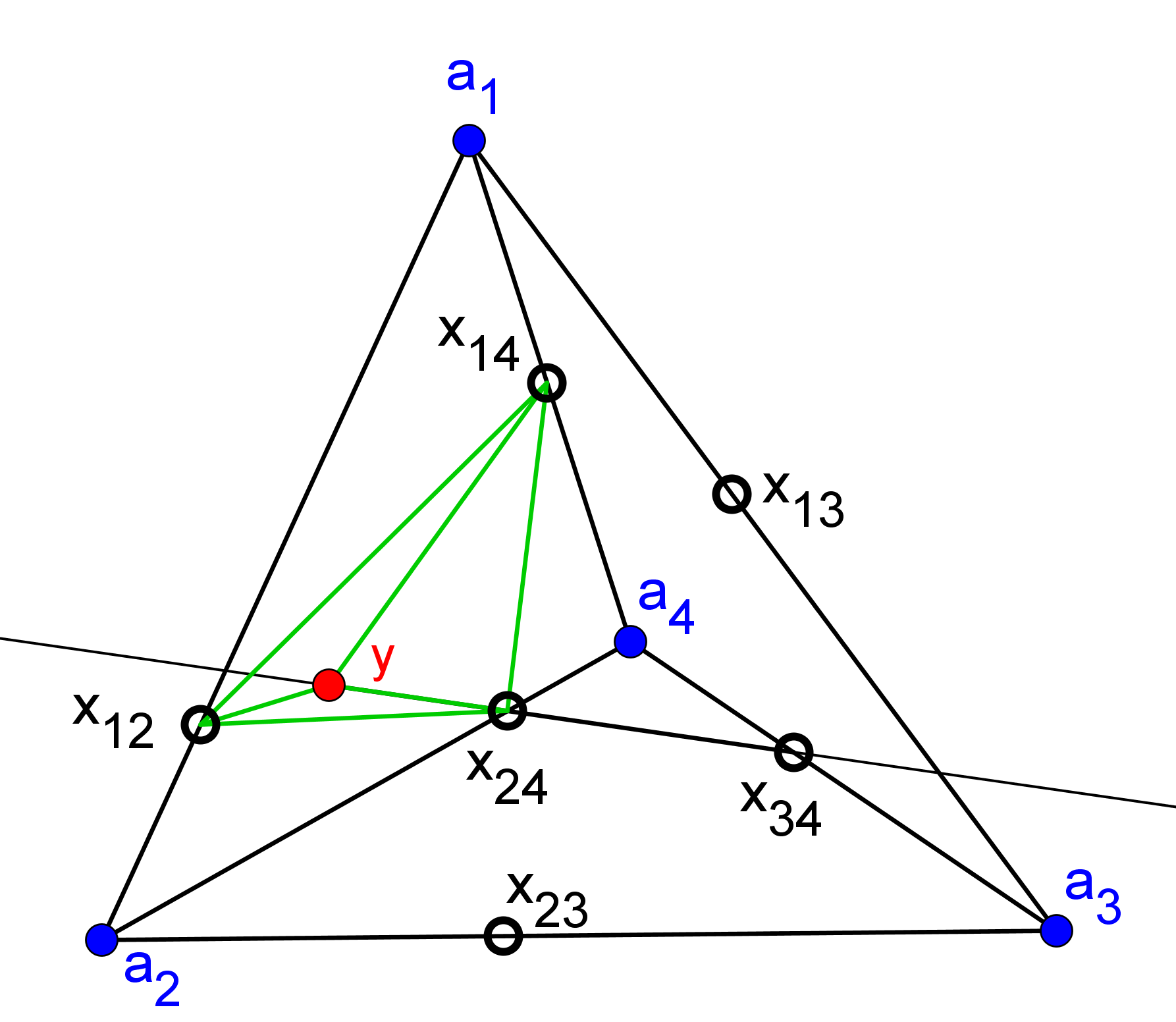}
\hspace*{-4mm}
  \includegraphics[scale=0.69]{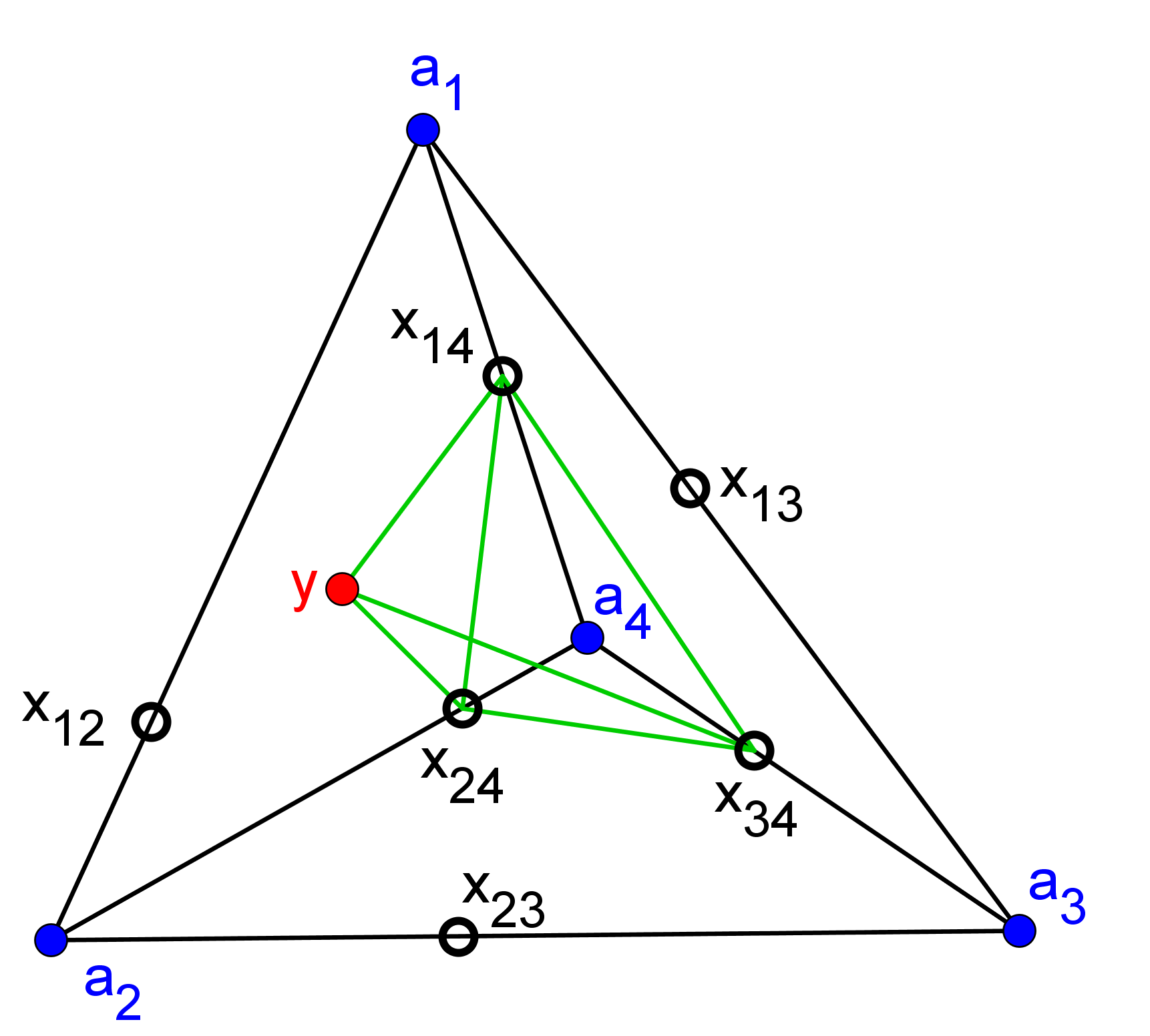}
\hspace*{-4mm}\\
\vspace*{-2ex}
    \caption{Diagrams of $\{4,3,2,2\}$ cases 1, 2, and 3.}
  \end{figure}

  \begin{enumerate}
  \item If $y$ blocks two points $x_{i4}$ and $x_{j4}$, then suppose
    WLOG that $x_{24}*y*x_{34}$.  Now $y$ is on one side or the other
    of $\overleftrightarrow{x_{14}a_4}$, and therefore sees at least
    one of $x_{12}$, $x_{13}$; WLOG we may say that $y$ sees $x_{12}$.
    Now $x_{12}$, $x_{14}$, $x_{24}$, and $y$ are four mutually
    visible points in $Q'$.
  \item If $y$ is collinear with, but does not block, two points
    $x_{i4}$ and $x_{j4}$, then suppose WLOG that $y*x_{24}*x_{34}$.
    Now $x_{12}$ is on one side or the other of
    $\overleftrightarrow{yx_{24}}$, and therefore sees either $x_{14}$
    or $x_{23}$.  Whichever one $x_{12}$ sees, that point is mutually
    visible with $x_{12}$, $x_{24}$, and $y$; thus we have four
    mutually visible points in $Q'$.
  \item If $y$ is not collinear with two points $x_{i4}$ and $x_{j4}$,
    then $y$ sees all such points; thus $x_{14}$, $x_{24}$, $x_{34}$,
    and $y$ are four mutually visible points in $Q'$.
  \end{enumerate}
  In all cases, we have four mutually visible points in three colours,
  which is impossible.

  Now suppose that $Q$ is $\{4,3,3,2\}$-blocked. Then $Q'$ comprises six
  points $x_{ij}$ and two additional points, $y_1$ and $y_2$.  Let
  $\midx = \{ x_{i4}:1\le i \le 3\}$ and let $\midx_2$ be the set of
  segments with endpoints in $\midx$.  How many of the points $y_i$
  block members of $\midx_2$?  Again, we have three cases:

  \begin{figure}[ht]
    \includegraphics[scale=0.69]{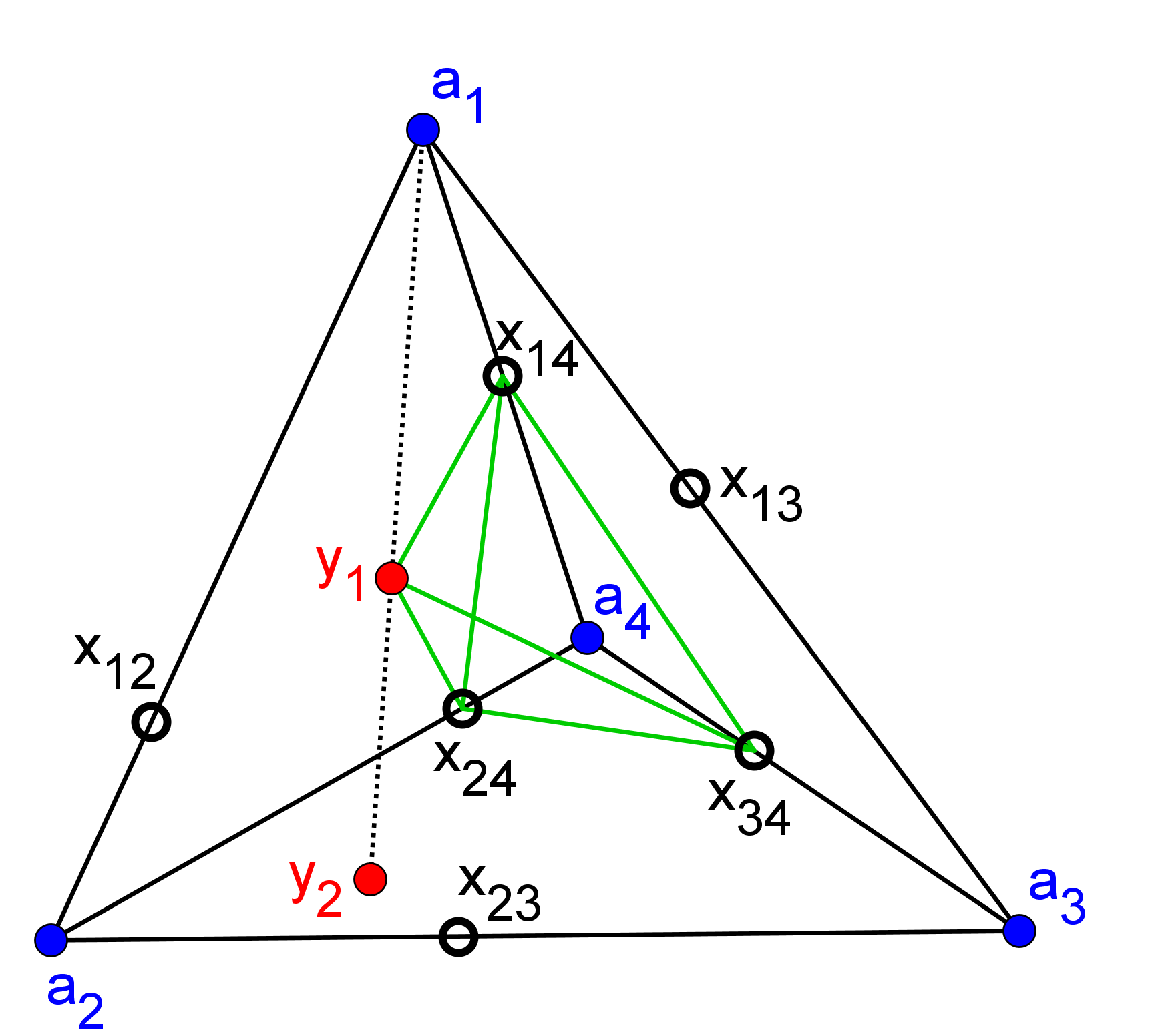}
    \hfill
    \includegraphics[scale=0.69]{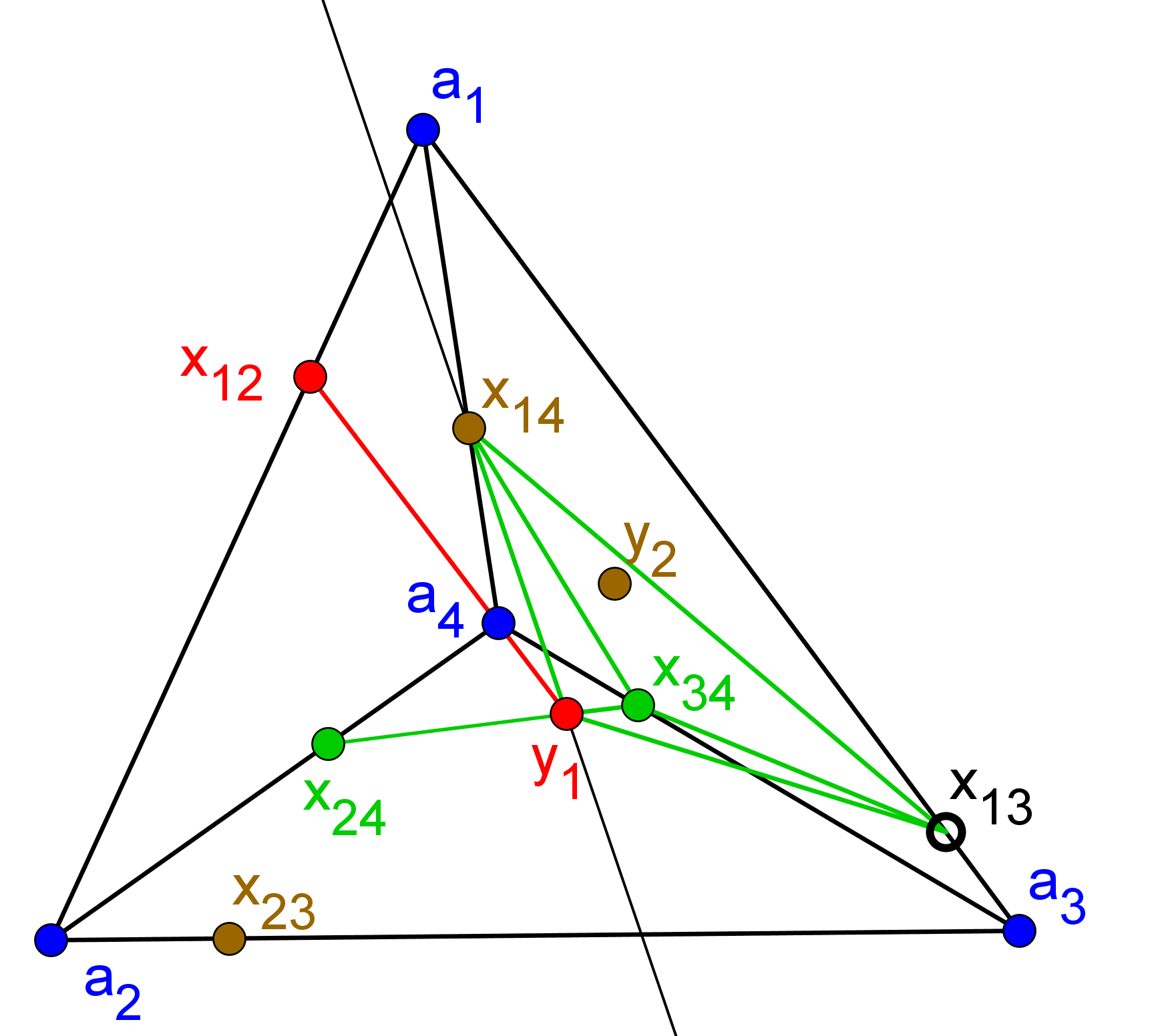}
    \hfill
    \includegraphics[scale=0.69]{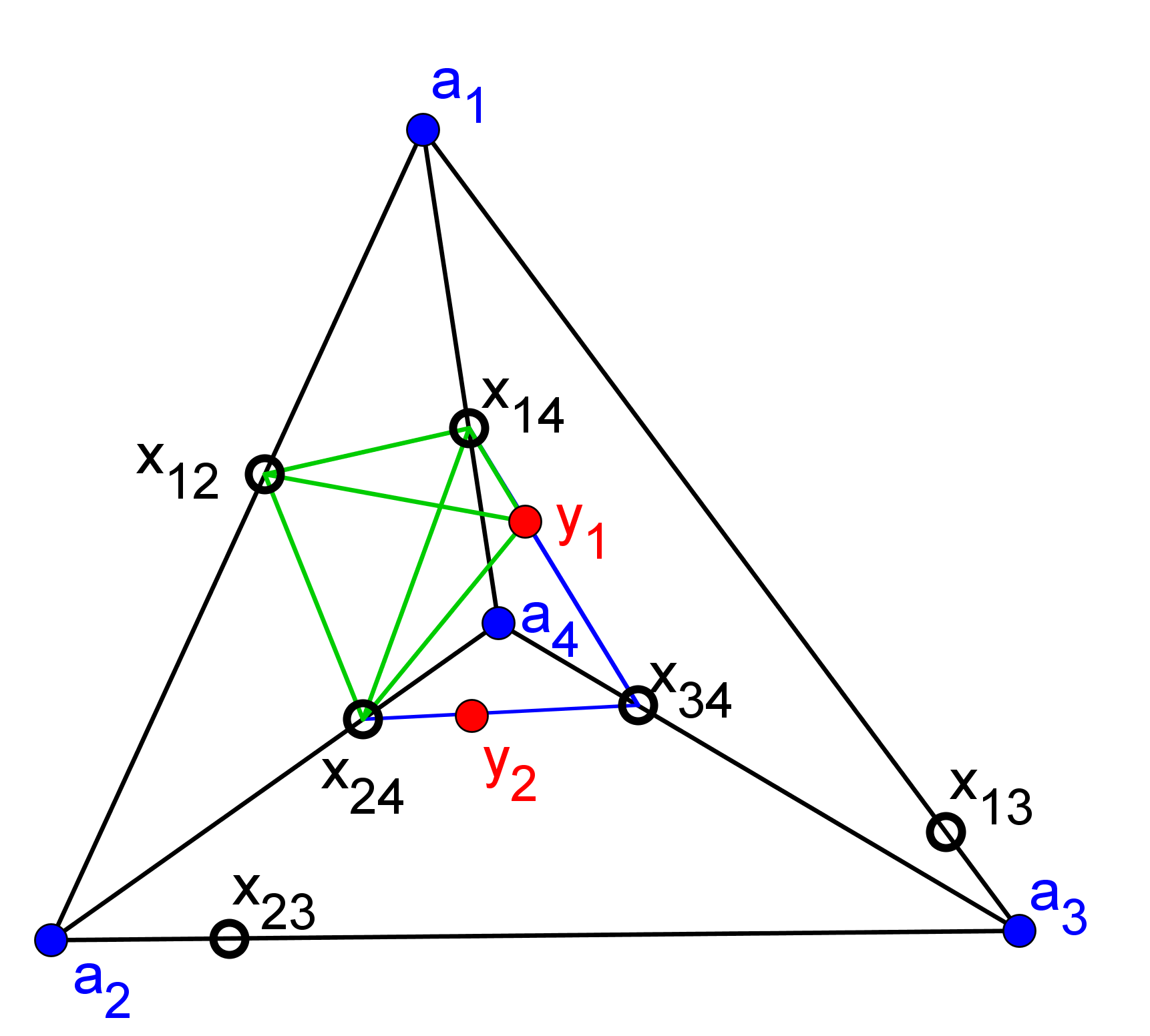}
\hspace*{-2mm}\\
\vspace*{-2ex}
   \caption{Diagrams of $\{4,3,3,2\}$ cases 1, 2, and 3.}
  \end{figure}

  \begin{enumerate}
  \item If neither $y_1$ nor $y_2$ blocks a segment in $\midx_2$, then
    recall that $y_1$ and $y_2$ cannot both block each other.  Thus,
    WLOG, $y_2$ does not block $y_1$.  Now $\midx$ and $y_1$ give us
    four mutually visible points in $Q'$.

  \item If $y_1$ blocks a segment in $\midx_2$ and $y_2$ does not,
    then WLOG $x_{24}*y_1*x_{34}$. \0{ If $y_2$ and $x_{12}$ are on
      opposite sides of $\overleftrightarrow{x_{14}y_1}$, then
      $x_{12}$, $x_{14}$, $x_{24}$, and $y_1$ are four mutually
      visible points in $Q'$.  No: $a_4$ could block
      $\overline{y_1x_{12}}$. On the other hand,}Also WLOG, suppose
    that $a_2$ and $a_4$ are on the same side of
    $\overleftrightarrow{x_{14}y_1}$, so $a_4 \notin \conv(x_{13},
    x_{14}, x_{34}, y_1)$.  Now $y_2 \in \conv(x_{13}, x_{14}, x_{34},
    y_1)$, for otherwise these four points in $Q'$ are mutually
    visible.

    Also note that $\overline{y_1x_{12}}$ is the only edge with
    points in $\{x_{12}, x_{14}, x_{24}, y_1\}$ that can be blocked
    by $a_4$.  To prevent these four points from being mutually
    visible, we require $y_1*a_4*x_{12}$.

    Thus $y_1$ and $x_{12}$ belong to the same colour class, which we
    may suppose is $B$.  Since $x_{12}$, $x_{14}$, and $x_{24}$ are
    mutually visible, we may suppose that $x_{24}\in C$ and
    $x_{14}\in D$.

    Since $x_{24}*y_1*x_{34}$, we have $x_{34}\in C$.  Since $x_{23}$
    sees $x_{24}$ and $y_1$, we have $x_{23}\in D$.

    Now $y_2$ sees $y_1 \in B$ and $x_{34}\in C$, and we conclude
    that $y_2\in D$; however, $y_2$ sees either $x_{14}$ or $x_{23}$,
    and colour class $D$ is not blocked.

  \item If both $y_1$ and $y_2$ block segments in $\midx_2$, then WLOG
    $x_{i4}*y_i*x_{34}$ for $i=1,2$.  Now $\{x_{14}, x_{24}, a_4\}$ is
    insufficient to block $x_{12}$ from $\{y_1, y_2\}$, for this would
    imply some $x_{12}*x_{i4}*y_i*x_{34}$.  Therefore some $y_i$ sees
    $x_{12}$; now $x_{12}$, $x_{14}$, $x_{24}$, and $y_i$ are four
    mutually visible points in $Q'$.
  \end{enumerate}

  The only remaining case is that $Q$ is $\{4,2,2,2\}$-blocked. This is
  possible, as illustrated in \figref{K4222}. We now show that this
  point set is essentially the only $\{4,2,2,2\}$-blocked set, up to
  betweenness-preserving deformations. We have exactly enough points
  in colour classes $B$, $C$, and $D$ to block all edges between
  points in colour class $A$.  As each of the three $A$-triangles
  with point $a_4$ must contain a representative from each of the
  other three colour classes, it follows that of the six $A$-edges,
  $B$ blocks one interior edge and the opposite boundary edge, and
  likewise for $C$ and $D$.  WLOG: $b_1=x_{14}$, $b_2=x_{23}$,
  $c_1=x_{24}$, $c_2=x_{13}$, $d_1=x_{34}$, $d_2=x_{12}$.  Since $b_1$
  blocks $a_4$, $a_4$ cannot block $b_1$.  Since $\overline{b_1b_2}$
  can be blocked only by an interior point of $\conv(A)$, it follows
  that either $c_1$ or $d_1$ blocks $\overline{b_1b_2}$.  As these
  cases are symmetric, we may choose $c_1\in\overline{b_1b_2}$.  Now
  $b_1$ cannot block $c_1$, so $c_1$ must be blocked by $d_1$, which
  must in turn be blocked by $b_1$.

  Now we show that $P=Q$. (This basically says that the point set in
  \figref{K4222} cannot be extended without introducing a new
  colour.)\ Suppose to the contrary that some point $x$ is in $P
  \setminus Q$; thus $x \notin \conv(A)$.  Note that every point
  outside $\conv(A)$ can see a vertex of $\conv(A)$, so $x$ must not
  be in colour class $A$.  WLOG, we suppose $x$ matches the points in
  $B$.  Recall that $b_2$ is on the supporting line
  $\overleftrightarrow{a_2a_3}$, and will see $x$ unless $x$ is in the
  same half-plane (as determined by this line) as the rest of $Q$.
  Who blocks $\overline{b_1x}$?  Not $a_2$ or $a_3$, for this would
  put $x$ in the wrong half-plane.  Not $a_1$, $a_4$, $d_1$, or $d_2$,
  for $b_1$ blocks each of these.  Not $c_1$, which is on
  $\overline{b_1b_2}$.  Therefore $\overline{b_1x}$ can be blocked
  only by $c_2$.  Since $c_1\in\overline{b_1b_2}$ and
  $c_2\in\overline{b_1x}$, it follows that $b_2$ and $x$ (and any
  blocker between them) are on the same side of
  $\overleftrightarrow{c_1c_2}$.  But the only other points of $Q$ in
  that open half-plane are $a_2$ and $a_3$, which cannot block $b_2$.
  Thus $x$ sees $b_2$, which is a contradiction. Thus $P=Q$ and $P$ is
  $\{4,2,2,2\}$-blocked.
\end{proof}

We now prove the main theorem of this section.

\begin{theorem}
A 4-set $\{a,b,c,d\}$ is representable  if and only if
  \begin{itemize}
  \item $\{a,b,c,d\} = \{4,2,2,1\}$, or
  \item $\{a,b,c,d\} = \{4,2,2,2\}$, or
  \item all of $a,b,c,d \leq 3$ except for $\{3,3,3,1\}$
  \end{itemize}
\end{theorem}

\begin{proof}
  \twofigref{BradsTwistedGrid}{K4222} respectively show
  $\{4,2,2,1\}$-blocked and $\{4,2,2,2\}$-blocked point sets.  When
  $a,b,c,d \leq 3$, the required constructions are described in
  \propref{SmallColourClasses}.  Now we prove that these are the only
  representable 4-sets.  Let $P$ be a $4$-blocked point set.  By
  \lemref{FourInColourClass}, each colour class has at most four
  points. Let $S$ be the largest colour class.  If $|S|\leq 3$ then we
  are done by \propref{SmallColourClasses}.  Now assume that $|S|=4$.
  If $S$ is in nonconvex position, then $P$ is $\{4,2,2,2\}$-blocked by
  \lemref{K4222}.  If $S$ is in convex position, then $P$ is
  $\{4,2,2,1\}$-blocked by \corref{K4221}.
\end{proof}

\begin{corollary} 
  Every 4-blocked set has at most 12 points, and there is a
  4-blocked set with 12 points.
\end{corollary}

Note that in addition to the $\{3,3,3,3\}$-blocked set shown in
\figref{K3333}, there is a different $\{3,3,3,3\}$-blocked point set, as
illustrated in \figref{3333b}.

\begin{figure}[!ht]
  \includegraphics[scale=0.6]{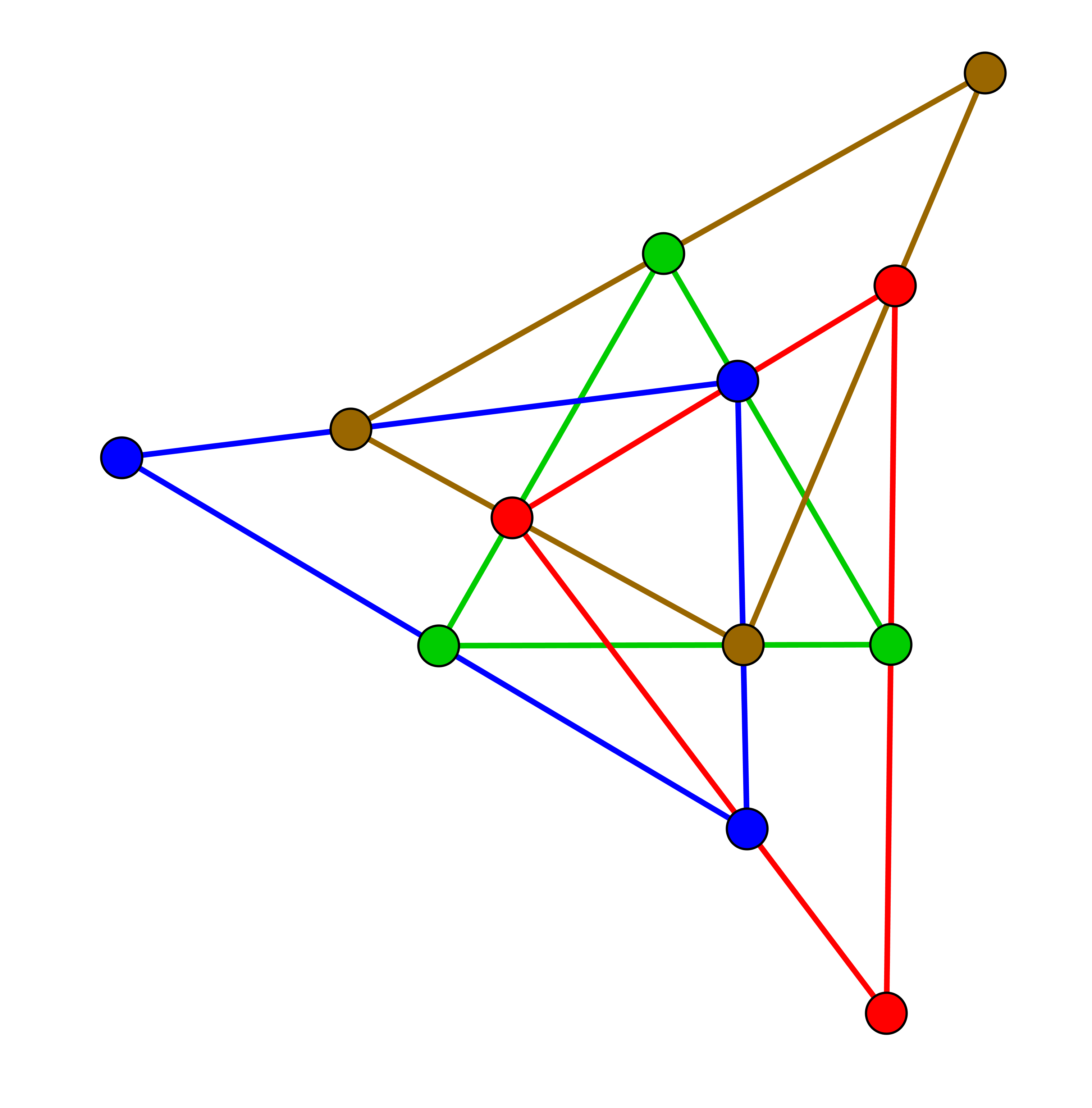}
\vspace*{-2ex}
  \caption{\figlabel{3333b} Another $\{3,3,3,3\}$-blocked point
    set.}
\end{figure}

\section{Midpoint-Blocked Point Sets}
\seclabel{Midpoint}

A $k$-blocked point set $P$ is \emph{$k$-midpoint-blocked} if for each
monochromatic pair of distinct points $v,w\in P$ the midpoint of
$\overline{vw}$ is in $P$.  Of course, the midpoint of $\overline{vw}$
blocks $v$ and $w$. A point set $P$ is
\emph{$\{n_1,\dots,n_k\}$-midpoint-blocked} if it is
$\{n_1,\dots,n_k\}$-blocked and $k$-midpoint-blocked. 
For example, the
point set in \figref{K3333} is $\{3,3,3,3\}$-midpoint-blocked.  

Another interesting example is the projection\footnote{If $G$ is the
  visibility graph of some point set $P\subseteq\R^d$, then $G$ is the
  visibility graph of some projection of $P$ to $\R^2$ (since a random
  projection of $P$ to $\R^2$ is occlusion-free with probability
  $1$).} of $[3]^d$. With $d=1$ this point set is $\{2,1\}$-blocked,
with $d=2$ it is $\{4,2,2,1\}$-blocked, and with $d=3$ it is
$\{8,4,4,4,2,2,2,1\}$-blocked.  In general, each set of points with
exactly the same set of coordinates equal to $2$ is a colour
class, there are $2^{d-i}$ colour classes of points with exactly $i$
coordinates equal to $2$, and $[3]^d$ is $\{\binom{d}{i}\times
2^i:i\in[0,d]\}$-midpoint-blocked and $2^d$-midpoint-blocked.

We now prove \conjref{kBlocked} when restricted to
$k$-midpoint-blocked point sets. (Finally we have weakened
\conjref{BigClique} to something proveable!)\

\citet{HHUZ} introduced the following definition. Let $m(n)$ be the
minimum number of midpoints determined by some set of $n$ points in
general position in the plane. Since the midpoint of $\overline{vw}$
blocks $v$ and $w$, we have $b(n)\leq m(n)$. \citet{HHUZ} constructed
a set of $n$ points in general position in the plane that determine at
most $cn^{\log 3}$ midpoints for some contant $c$. (All logarithms
here are binary.)\ Thus $b(n)\leq
m(n)\leq cn^{\log3} =cn^{1.585\ldots}$. This upper bound was
improved by \citet{Pach-Midpoints-Geom03} (and later by
\citet{Matousek09}) to
$$b(n)\leq m(n)\leq nc^{\sqrt{\log n}}\enspace.$$ 
\citet{HHUZ} conjectured that $m(n)$ is super-linear, which was
verified by \citet{Pach-Midpoints-Geom03}; that is,
$\frac{m(n)}{n}\rightarrow\infty$ as $n\rightarrow\infty$.
\citet{PorWood-Blockers} proved the following more precise version:
For some constant $c>0$, for all $\epsilon>0$ there is an integer
$N(\epsilon)$ such that $m(n)\geq c n(\log n)^{1/(3+\epsilon)}$ for
all $n\geq N(\epsilon)$.

\begin{theorem} 
  For each $k$ there is an integer $n$ such that every
  $k$-midpoint-blocked set has at most $n$ points. More precisely,
  there is an absolute constant $c$ and for each $\epsilon>0$ there is
  an an integer $N(\epsilon)$, such that for all $k$, every
  $k$-midpoint-blocked set has at most
  $k\max\{N(\epsilon),c^{(k-1)^{3+\epsilon}}\}$ points.
\end{theorem}

\begin{proof}
  Let $P$ be $k$-midpoint-blocked set of $n$ points. If $n\leq k
  N(\epsilon)$ then we are done. Now assume that $\frac{n}{k}>
  N(\epsilon)$.  Let $S$ be a set of exactly $s:=\ceil{\frac{n}{k}}$
  monochromatic points in $P$.  Thus $S$ is in general position by
  \lemref{GenPos}.  And for every pair of distinct points $v,w \in S$
  the midpoint of $\overline{vw}$ is in $P-S$.  Thus
$$ c \tfrac{n}{k}(\log\tfrac{n}{k})^{1/(3+\epsilon)}\leq m(s) \leq
n-s\leq n(1-\tfrac{1}{k})\enspace.$$ Hence
$(\log\tfrac{n}{k})^{1/(3+\epsilon)}\leq(k-1)/c$, implying
$n\leq k 2^{((k-1)/c)^{3+\epsilon}}$. The result follows.
\end{proof}




We now construct $k$-midpoint-blocked point sets with a `large' number
of points.  The method is based on the following product of point sets
$P$ and $Q$.  For each point $v\in P\cup Q$, let $(x_v,y_v)$ be the
coordinates of $v$.  Let $P\times Q$ be the point set $\{(v,w):v\in
P,w\in Q\}$ where $(v,w)$ is at $(x_v,y_v,x_w,y_w)$ in $4$-dimensional
space. For brevity we do not distinguish between a point in $\R^4$ and
its image in an occlusion-free projection of the visibility graph of
$P\times Q$ into $\R^2$.

\begin{lemma}
  \lemlabel{ProductLemma} If $P$ is a
  $\{n_1,\dots,n_k\}$-midpoint-blocked point set and $Q$ is a
  $\{m_1,\dots,m_{\ell}\}$-midpoint-blocked point set, then $P\times Q$
  is $\{n_im_j:i\in[k],j\in[\ell]\}$-midpoint-blocked.
\end{lemma}

\begin{proof}
  Colour each point $(v,w)$ in $P\times Q$ by the pair
  $(\col(v),\col(w))$.  Thus there are $n_im_j$ points for the
  $(i,j)$-th colour class.  Consider distinct points $(v,w)$ and
  $(a,b)$ in $P\times Q$.

  Suppose that $\col(v,w)=\col(a,b)$. Thus $\col(v)=\col(a)$ and
  $\col(w)=\col(b)$.  Since $P$ and $Q$ are midpoint-blocked,
  $\half(v+a)\in P$ and $\half(w+b)\in Q$.  Thus
  $(\half(v+a),\half(w+b))$, which is positioned at
  $(\half(x_v+x_a),\half(y_v+y_a),\half(x_w+x_b),\half(y_w+y_b))$, is
  in $P\times Q$.  This point is the midpoint of
  $\overline{(v,w)(a,b)}$.  Thus $(v,w)$ and $(a,b)$ are blocked by
  their midpoint in $P\times Q$.

  Conversely, suppose that some point $(r,s)\in P\times Q$ blocks
  $(v,w)$ and $(a,b)$.  Thus $x_r=\alpha x_v+(1-\alpha)x_a$ for some
  $\alpha\in(0,1)$, and $y_r=\beta y_v+(1-\beta)y_a$ for some
  $\beta\in(0,1)$, and $x_s=\delta x_w+(1-\delta)x_b$ for some
  $\delta\in(0,1)$, and $y_s=\gamma y_w+(1-\gamma)y_b$ for some
  $\gamma\in(0,1)$.  Hence $r$ blocks $v$ and $a$ in $P$, and $s$
  blocks $w$ and $b$ in $Q$.  Thus $\col(v)=\col(a)\neq\col(r)$ in
  $P$, and $\col(w)=\col(b)\neq\col(s)$ in $Q$, implying
  $(\col(v),\col(w))=(\col(a),\col(b))$.

  We have shown that two points in $P\times Q$ are blocked if and only
  if they have the same colour.  Thus $P\times Q$ is blocked.  Since
  every blocker is a midpoint, $P\times Q$ is midpoint-blocked.
\end{proof}

Say $P$ is a $k$-midpoint blocked set of $n$ points.  By
\lemref{ProductLemma}, the $i$-fold product $P^i:=P\times\dots\times
P$ is a $k^i$-blocked set of $n^i=(k^i)^{\log_k n}$ points. Taking $P$
to be the $\{3,3,3,3\}$-midpoint-blocked point set in \figref{K3333}, we
obtain the following result:

\begin{theorem}
  For all $k$ a power of $4$, there is a $k$-blocked set of $k^{\log_4
    12}=k^{1.79\ldots}$ points.
\end{theorem}

This result describes the largest known construction of $k$-blocked or
$k$-midpoint blocked point sets.  To promote further research, we make
the following strong conjectures:

\begin{conjecture}
  \conjlabel{SizeBlocked} Every $k$-blocked point set has $O(k^2)$
  points.
\end{conjecture}

\begin{conjecture}
  \conjlabel{SizeColorClass} In every $k$-blocked point set there are
  at most $k$ points in each colour class.
\end{conjecture}

\conjref{SizeColorClass} would be tight for the projection of $[3]^d$
with $k=2^d$.  Of course, \conjref{SizeColorClass} implies
\conjref{SizeBlocked}, which implies \conjref{kBlocked}.

\section*{Acknowledgements} This research was initiated at The 24th
Bellairs Winter Workshop on Computational Geometry, held in February
2009 at the Bellairs Research Institute of McGill University in
Barbados. The authors are grateful to Godfried Toussaint and Erik
Demaine for organising the workshop, and to the other 
participants for providing a stimulating working environment.


\def\cprime{$'$} \def\soft#1{\leavevmode\setbox0=\hbox{h}\dimen7=\ht0\advance
  \dimen7 by-1ex\relax\if t#1\relax\rlap{\raise.6\dimen7
  \hbox{\kern.3ex\char'47}}#1\relax\else\if T#1\relax
  \rlap{\raise.5\dimen7\hbox{\kern1.3ex\char'47}}#1\relax \else\if
  d#1\relax\rlap{\raise.5\dimen7\hbox{\kern.9ex \char'47}}#1\relax\else\if
  D#1\relax\rlap{\raise.5\dimen7 \hbox{\kern1.4ex\char'47}}#1\relax\else\if
  l#1\relax \rlap{\raise.5\dimen7\hbox{\kern.4ex\char'47}}#1\relax \else\if
  L#1\relax\rlap{\raise.5\dimen7\hbox{\kern.7ex
  \char'47}}#1\relax\else\message{accent \string\soft \space #1 not
  defined!}#1\relax\fi\fi\fi\fi\fi\fi} \def\Dbar{\leavevmode\lower.6ex\hbox to
  0pt{\hskip-.23ex\accent"16\hss}D}

\end{document}